\documentclass[12pt]{article}

\usepackage{amsmath,amssymb}
\usepackage[english]{babel}
\usepackage[utf8]{inputenc}
\usepackage{epsfig}
\usepackage{graphicx}
\usepackage[all]{xy}
\usepackage{lmodern}
\usepackage{stackengine}
\usepackage{latexsym, amsmath, amssymb, amsthm, mathrsfs, framed}
\usepackage{paralist}
\usepackage[colorlinks=true,linkcolor=blue,urlcolor=blue, citecolor=blue]{hyperref}
\usepackage[shortalphabetic]{amsrefs}
\usepackage{bbm}
\usepackage{mathptmx}
\usepackage{microtype}
\usepackage{enumitem}
\usepackage{mathtools}
\usepackage{comment}

\DeclareMathAlphabet{\mathcal}{OMS}{cmsy}{m}{n}

\newtheorem{teo}{Theorem}[]
\newtheorem*{teo*}{Theorem}

\newtheorem{cor}{Corolary}
\newtheorem{prop}{Proposition}

\newtheorem{lemma}{Lemma}

\theoremstyle{remark}
\newtheorem{rmk}[teo]{Remark}

\newcommand{\vol}{\mathop{\rm vol}}    

\newcommand{\Ker}{\mathop{\rm Ker}}

\newcommand{\C}{{\mathbb C}}

\newcommand{\N}{{\mathbb N}}
\newcommand{\R}{{\mathbb R}}

\newcommand{\Z}{{\mathbb Z}}

\renewcommand{\k}{{\mathrm k}}

\renewcommand{\imath}{{\bf i}}

\newcommand{\tor}{\xymatrix{\ar@{-->}[r]&}}
\newcommand{\mapstor}{\xymatrix{\ar@{|-->}[r]&}}

\DeclareMathOperator{\Grass}{\mathbb G}

\newcommand{\Grassc}[2]{\ensuremath{\Grass({#1},{#2})}}

\title{Average conditioning of underdetermined polynomial systems}

\begin{document}

\newcommand{\Pn}{{\mathbb P(\mathbb C^{n+1})}}
\newcommand{\Cm}{{\mathbb C^m}}
\newcommand{\Cn}{{\mathbb C^n}}

\author{
Federico Carrasco\\
Universidad de La República\\
URUGUAY\\
\normalsize e-mail: {\tt fcarrasco@cmat.edu.uy}
}

\maketitle
\begin{abstract} This article study the average conditioning for a random underdetermined polynomial system. The expected value of the moments of the condition number are compared to the moments of the condition number of random matrices. An expression for these moments is given by studying the kernel finding problem for random matrices. Furthermore, the second moment of the Frobenius condition number is computed. 
\end{abstract}  


\section{Introduction}
Solving systems of equations is a fundamental problem, which has been deeply studied from different points of view, such as algebraic, geometric and numerical approaches.

A classic numerical method of solving such systems, is the called Newton’s iteration.
In this article, we establish the average values of a key quantity influencing the computational performance of Newton's operator in underdetermined scenarios. 
Shub and Smale introduced Newton's operator for underdetermined systems of equations in their work \cite{ShubSmale_IV} (cf Dégot \cite{Degot}). 
The primary objective of their efforts was to develop and analyze effective algorithms for computing approximations to complete intersection algebraic subvarieties of $\C^n$.

This key quantity is the condition number, which measures the sensitivity of the set of solutions of the considered system, to variations of the equations (see Blum et al. \cite{BCSS}, Bürgisser-Cucker \cite{BurCuck}).

The condition number was introduced by Turing \cite{Turing} and von Neuman-Goldstine \cite{NeuGold}, while studying the propagation of errors for linear equation solving and matrix inversion.
Ever since then, condition numbers have played a leading role in the study of both accuracy and complexity of numerical algorithms.

As pointed out by Demmel \cite{Demmel}, computing the condition number of any numerical problem is a time-consuming task that suffers from intrinsic stability problems.
For this reason, understanding the behavior of the condition number in such a way that we can rely on probabilistic arguments is a useful strategy.

In order to be more precise in our statement we need to introduce some preliminary notations.

\subsection{Preliminaries}
  
For every positive integer $d\in\N$, let $\mathcal H^n_d$ be the complex vector space of all homogeneous polynomials of degree $d$ in $(n+1)$-complex variables with coefficients in $\C$.

We denote by a multi-index $j\coloneqq (j_0,\cdots,j_n)\in\Z^{n+1}$, $j_i\geq 0$ for $i=0,\cdots,n$, and consider $|j|=j_0+\cdots+j_n$.
Then, for $x=(x_0,\cdots,x_n)\in\C^{n+1}$, we write
$$
x^j \coloneqq x_0^{j_0}\cdots x_n^{j_n}.
$$

We consider the Bombieri-Weyl Hermitian product in $\mathcal H^n_d$, defined as follows.
Let $h,g\in \mathcal H^n_d$, be two elements, $h(x)=\displaystyle\sum_{|j|=d}a_j x^j$, $g(x)=\displaystyle\sum_{|j|=d}b_j x^j$, we define 
$$
\langle h,g\rangle_d=\sum_{|j|=l}a_j\overline b_j {d\choose j}^{-1},
$$
where $\displaystyle{d\choose j} = \displaystyle\frac{d!}{j_0!\cdots j_n!}$ (see Shub-Smale \cite{ShubSmale}).

\medskip

For any list of positives degrees $(d)\coloneqq(d_1,\cdots,d_r)$, $r\leq n$, let
$$
\mathcal H^{r,n}_{(d)}\coloneqq \prod_{i=1}^r\mathcal H^n_{d_i}
$$
be the complex vector space of homogeneous polynomial systems $h\coloneqq (h_1,\cdots,h_r)$ of respective degrees $d_i$.

We denote by $\mathcal D_r$ the Bézout number associated with the list $(d)$, i.e.
$$
\mathcal D_r\coloneqq \prod_{i=1}^r d_i.
$$

The previously defined Hermitian product induces a Hermitian product in $\mathcal H^{r,n}_{(d)}$ as follows.
For any two elements $h=(h_1,\cdots,h_r)$, $g=(g_1,\cdots,g_r) \in\mathcal H^{r,n}_{(d)}$, we define
$$
\langle h,g\rangle\coloneqq\sum_{i=1}^r\langle h_i,g_i\rangle_{d_i}.
$$

The Hermitian product $\langle\cdot,\cdot\rangle$ induces a Riemannian structure in the space $\mathcal H^{r,n}_{(d)}$.

The space $\C^{n+1}$ is equipped with the canonical Hermitian inner product $\langle\cdot ,\cdot\rangle$ which induces the usual euclidean norm $\|\cdot\|$, and we denote by $\Pn$ its associated projective space. 
This is a smooth manifold which carries a natural Riemannian metric, namely, the real part of the Fubini-Study metric on $\Pn$ given in the following way: for a non zero $x\in\mathbb C^{n+1}$,
$$
\langle w,w'\rangle_x:=\frac{\langle w,w'\rangle}{\|x\|^2},
$$
for all $w$, $w'$ in the Hermitian complement $x^\perp$ of $x$. This induces the norm $\|\cdot\|_x$ in $T_x\Pn$.

The space $\mathcal H^{r,n}_{(d)}\times\Pn$ is endowed with the Riemannian product structure 
(see Blum et al. \cite{BCSS}).

\subsection{Condition Number}

The condition number associated to a computational problem measures the sensitivity of the outputs of the considered problem, to variations of the input (see Bürgisser-Cucker \cite{BurCuck}). 
In \cite{Dedieu} Dedieu defined the condition number of a polynomial $f:\C^n\to\C$ at a point $x\in\C^n$, such that $f(x)=0$ and $Df(x)$ is surjective, as 
$$
\mu(f,x)\coloneqq\|Df(x)^\dagger\|_{op},
$$
where $Df(x)^\dagger$ is the Moore-Penrose pseudo inverse of the linear map $Df(x)$, i.e. the derivative of $f$ at $x$, and $\|Df(x)^\dagger\|_{op}$ is the operator norm of $Df(x)^\dagger$. 

Following this idea, and using the normalized condition number $\mu_{norm}$ introduced in Shub-Smale \cite{ShubSmale}, Dégot \cite{Degot} suggested an extension of this condition number for the undetermined case which was adjusted into a projective quantity by Beltrán-Pardo in \cite{BelParUnder}. 

As done in \cite{BelShub}, we will consider the Frobenius condition number, just by considering the Frobenius norm insted of the operator one.

Given $h\in\mathcal H^{r,n}_{(d)}$ and $x\in\mathbb C^{n+1}$ such that $h(x)=0$ and $Dh(x)$ has rank $r$, then the Frobenius condition number of $h$ at $x$ is defined by
\begin{equation}\label{eq:muf}
  \mu^r_F(h,x)\coloneqq\|h\|\|Dh(x)^\dagger\Delta(d_i^{1/2}\|x\|^{d_i-1})\|_F,  
\end{equation}
where $\|\cdot\|_F$ is the Frobenius norm (i.e. $\mathrm{trace}(L^*L)^{1/2}$ where $L^*$ is the adjoint of $L$).
If the rank of $Dh(x)$ is strictly smaller than $r$, we set $\mu^r_F(h,x)\coloneqq\infty$.

When $r=n$, we will write $\mu_{F}(h,x)$ instead.

\medskip

Let $\Sigma'\coloneqq\{(h,x)\in\mathcal H^{r,n}_{(d)}\times\mathbb C^{n+1} \ : \ h(x)=0; \ \mathrm{rank}(Dh(x))<r \}$ and $\Sigma\subseteq\mathcal H^{r,n}_{(d)}$ be the projection of $\Sigma'$ onto the first coordinate, commonly referred as the discriminant variety.
Observe that for all $h\in\mathcal H^{r,n}_{(d)}\setminus\Sigma$, thanks to the inverse image of a regular value theorem, the zero set 
$$
V_h\coloneqq\{x\in\Pn \ : \ h(x)=0\},
$$
is a complex smooth submanifold of $\Pn$ of dimension $n-r$.
Then it is endowed with a complex Riemannian structure that induces a finite volume form.

Now, for $h\in\mathcal H^{r,n}_{(d)}\setminus\Sigma$ it makes sense to consider the $2$-nd moment of the Frobenius condition number of $h$, $\mu^{r,2}_{F,Av}(h)$, as the average of $(\mu^r_{F}(h,x))^2$ over its zero set $V_h$, i.e.
\begin{equation}\label{eq:muhat}
    \mu^{r,2}_{F,Av}(h)\coloneqq\frac{1}{\vol(V_h)}\int_{x\in V_h}\mu^r_{F}(h,x)^2 \, dV_h.    
\end{equation}

In this way, one can be free from the solutions of the system and just think of the
conditioning of the polynomial system itself.

\subsection{Main Result}

The main result of this article gives a closed formula for the expected value of $\frac{\mu^{r,2}_{F,Av}(h)}{\|h\|^2}$.
To be accurate in this notion, we need to fix a probability measure in $\mathcal H^{r,n}_{(d)}$. 

Consider the average with respect to the standard Gaussian distribution on $\mathcal H^{r,n}_{(d)}$, that is,
\begin{equation}\label{eq:gaussian}
  \underset{h\in\mathcal H^{r,n}_{(d)}}{\mathbb E}(\phi(h))=\frac{1}{\pi^N}\int_{h\in\mathcal H^{r,n}_{(d)}}\phi(h)e^{-\|h\|^2},  
\end{equation}
where $N$ is the complex dimension of $\mathcal H^{r,n}_{(d)}$ and $\phi:\mathcal H^{r,n}_{(d)}\to\R$ is a measurable function.

The main result of this article is the following.

\begin{teo}[Main Theorem]\label{relavg}
  The expected value, with respect to the standard Gaussian distribution, of the $2$-nd moment of the relative Frobenius condition number $\frac{\mu^{r,2}_{F,Av}(h)}{\|h\|^2}$ satisfies:
  $$
  \underset{h\in\mathcal H^{r,n}_{(d)}}{\mathbb E}\left(\frac{\mu^{r,2}_{F,Av}(h)}{\|h\|^2}\right)=\frac{r}{n-r+1}.
  $$
\end{teo}

As a matter of fact, we will be proving a more general result (see Section \ref{thmproof}).
That result can be extended to the case where we consider the operator norm.
Furthermore, after some computations (see Lemma \ref{invnor2mdet}), we get the closed expression for the case of the $2$-nd  moment stated in Theorem \ref{avg}. 
The proof of the general result strongly relies on Proposition \ref{exmuhat}, which states that the moments of the condition number for the polynomial case are essentially the moments of the condition number of a random matrix.  

\begin{rmk}
  Observe that by taking $r=n$ in the previous statement, one recovers the average of the $2$-nd moment of the relative Frobenius condition number for the determined case, namely
  $$
  \underset{g\in\mathcal H^{n,n}_{(d)}}{\mathbb E}\left(\frac{\mu^{2}_{F,Av}(g)}{\|g\|^2}\right)=n 
  $$
  (see Theorem 2 of Armentano et al. \cite{ABBCS}).
\end{rmk}

From Theorem \ref{relavg} and the previous remark we get the following proposition.

\begin{prop}\label{exp-under-det}
  The expected value, with respect to the standard Gaussian distribution, of $2$-nd moment of the relative Frobenius condition number satisfies
  $$
  \underset{h\in\mathcal H^{r,n}_{(d)}}{\mathbb E}\left(\frac{\mu^{r,2}_{F,Av}(h)}{\|h\|^2}\right)=\frac{1}{n-r+1}\underset{g\in\mathcal H^{r,r}_{(d)}}{\mathbb E}\left(\frac{\mu^2_{F,Av}(g)}{\|g\|^2}\right).
  $$
\end{prop}

This statement provides the expected value of the $2$-nd moment of the relative condition number for the underdetermined case in terms of the expected one in the determined case.
From a geometric perspective, these two cases exhibit notable distinctions, and there is no inherent requirement for these expected values to be in any kind of relation.   
It would be interesting to understand which are the reasons behind this relation.

\begin{cor}\label{avg}
  The expected value, with respect to the standard Gaussian distribution, of the $2$-nd moment of the Frobenius condition number $\mu^{r,2}_{F,Av}(h)$ satisfies:
  $$
  \underset{h\in\mathcal H^{r,n}_{(d)}}{\mathbb E}\left(\mu^{r,2}_{F,Av}(h)\right)=\frac{(N-1)r}{n-r+1},
  $$
  where $N$ is the dimension of $H^{r,n}_{(d)}$.
\end{cor}

\begin{rmk}
  In Theorem $1.4$ of Beltrán-Pardo \cite{BelParUnder}, an upper bound of the expected value of the average conditioning is computed, while using our argument we get an equality.
  Furthermore, using Cauchy-Schwartz inequality and Theorem \ref{avg}, we get a sharper bound. 
\end{rmk}

\section{Proof of Main Theorem}

For the proof of Theorem \ref{avg}, we need some background and previous results which are explained in the following subsections.

\subsection{Geometric framework}

In this section we set the geometric framework where we are going to work see Blum et al. \cite{BCSS} for details.

\medskip

Let $\mathcal V\subseteq \mathcal H^{r,n}_{(d)}\times\Pn$ be the solution variety, namely,
$$
\mathcal V\coloneqq\{(h,x)\in\mathcal H^{r,n}_{(d)}\times\Pn \ : \ h(x)=0\}.
$$

The solution variety $\mathcal V$ is a smooth connected subvariety of $\mathcal H^{r,n}_{(d)}\times\Pn$ of dimension $N+n-r$, whose tangent space at $(h,x)\in\mathcal V$ is the  set of pairs $(\dot h,\dot x)$ in $\mathcal H^{r,n}_{(d)}\times \C^{n+1}$ satisfying the following linear equations
$$
\dot h(x)+Dh(x)\dot x=0, \qquad x^*\dot x=0,
$$
where $x^*$ is the transposed conjugate of $x$. 

If we consider the two canonical projections $\pi_1:\mathcal V\to \mathcal H^{r,n}_{(d)}$, $\pi_2:\mathcal V\to\Pn$, we have the following diagram:

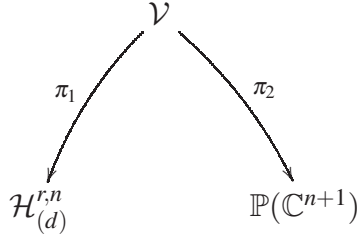
$$
\xymatrix{
&\mathcal{V}\ar@{->}@/_/[ddl]_{\pi_1}\ar@{->}@/^/[ddr]^{\pi_2}&\\
&&\\
\mathcal H^{r,n}_{(d)}&&\Pn}
$$

Let $\Sigma'\subseteq\mathcal V$ be the set of critical points of the projection $\pi_1$, i.e. the set of point $(h,x)$ such that $D\pi_1(h,x)$ is not surjective.
It can be proved that 
$$
\Sigma'=\{(h,x)\in\mathcal V \ : \ \mathrm{rank}(Dh(x))<r\}.
$$

We denote by $\Sigma\subseteq\mathcal H^{r,n}_{(d)}$ the image of $\Sigma'$ by $\pi_1$, it can be proved that it is a zero measure set.

Recall that for a homogeneous polynomial system $h\in\mathcal H^{r,n}_{(d)}$, we denote by $V_h\subseteq\Pn$ the zero set, namely, $V_h=\{x\in\Pn \ : \ h(x)=0\},$ which can be identified with $\pi_1^{-1}(h)$.
Observe that for all systems $h\in\mathcal H^{r,n}_{(d)}\setminus\Sigma$, the set $V_h$ is a projective variety of dimension $n-r$ (see Harris \cite{Harris}).

For every $x\in\Pn$ we denote by $V_x$ the linear subspace of $\mathcal H^{r,n}_{(d)}$ given as 
$$
V_x\coloneqq\{h\in\mathcal H^{r,n}_{(d)} \ : \ h(x)=0\},
$$
which can be identified with $\pi_2^{-1}(x)$.

The unitary group $U(n+1)$ acts on $\C^{n+1}$ as the group of linear automorphism that preserve the Hermitian product $\langle\cdot,\cdot\rangle$ on $\C^{n+1}$. 
More precisely,
$$
\langle\sigma v,\sigma w\rangle=\langle v,w\rangle, \ \text{for all} \ v, \ w\in\C^{n+1}, \  \sigma\in U(n+1).
$$
This action, induces an action of $U(n+1)$ on $\mathcal H^{r,n}_{(d)}$, given by 
$\sigma h(v)\coloneqq h(\sigma^{-1}v)$, for all $v\in\C^{n+1}$, and also induces a natural action of $U(n+1)$ on $\Pn$.
In conclusion, we have an action on the product $\mathcal H^{r,n}_{(d)}\times\Pn$ given by
$$
\sigma(h,x)\coloneqq(\sigma h,\sigma x).
$$
Then we have that $\mu^r_{F}:\mathcal V\setminus\Sigma'\to\mathbb R$ is unitarily invariant; that is, for all $\sigma\in U(n+1)$, $\mu^r_{F}(\sigma(h,x))=\mu^r_{F}(h,x)$ and $\Sigma'$ is unitarily invariant (see Beltrán-Shub \cite{BelShub}). 

\subsection{Double-fibration technique}

In this section we describe the double-fibration technique, which is of utmost importance for the computations done in Theorem \ref{avg}, see, for example, Shub-Smale\cite{ShubSmale}-\cite{ShubSmale_II}.

\medskip

Suppose $\Phi:M\to N$ is a surjective map from a Riemannian manifold $M$ to a Riemannian manifold $N$, whose derivative $D\Phi(x):T_xM\to T_{\Phi(x)}N$ is surjective for almost all $x\in M$.
The horizontal space $H_x\subset T_xM$ is defined as the orthogonal complement of $\Ker D\Phi(x)$.
The horizontal derivative of $\Phi$ at $x$ is the restriction of $D\Phi(x)$ to $H_x$.
The \textit{Normal Jacobian} $NJ_{\Phi}(x)$ is the absolute value of the determinant of the horizontal derivative, defined almost everywhere on $M$.

Then we have the following.

\begin{teo*}[\cite{BCSS}*{p. 241} The smooth coarea formula]\label{coarea}
  Let $M,N$ be Riemannian manifolds of respective dimension $m\geq n$, and let $\varphi:M\to N$ be a smooth surjective submersion.
  Then, for any positive measurable function $\Phi:M\to[0,+\infty)$ we have
  $$
  \int_{x\in M}\Phi(x)dM=\int_{y\in N}\int_{x\in\varphi^{-1}(y)}\frac{\Phi(x)}{NJ_\varphi(x)}d\varphi^{-1}(y)dN
  $$
  and
  $$
  \int_{x\in M}NJ_\varphi(x)\Phi(x)dM=\int_{y\in N}\int_{x\in\varphi^{-1}(y)}\Phi(x)d\varphi^{-1}(y)dN.
  $$
  \qed 
\end{teo*}

Consider the linear map $L_0:V_{e_0}\to\mathbb C^{r\times n}$ given by
\begin{align}\label{eq:l0}
L_0(h)=\Delta(d_i^{-1/2})Dh(e_0)|_{e_0^\perp}, 
\end{align}
where $\Delta(a_i)$ is the diagonal matrix whose diagonal entries are exactly the $a_i$'s. 
Observe that for any $g\in\Ker(L_0)^\perp$ we have that $\|g\|=\|L_0(g)\|_F$, where $\|L_0(g)\|_F$ is the Frobenius norm of the linear operator $L_0(g)$.
This implies that $NJ_{L_0}(g)=1$.
So, if we apply the smooth coarea formula to $L_0:V_{e_0}\to\mathbb C^{r\times n}$, we have that for any measurable mapping $\varphi:V_x\to[0,+\infty)$

\begin{equation}\label{eq:coarealx}
  \int_{h\in V_{e_0}}\varphi(h)dV_{e_0}=\int_{A\in\mathbb C^{r\times n}}\int_{h\in L_0^{-1}(A)}\varphi(h)dL_0^{-1}(A) \ dA.  
\end{equation} 
 
The double-fibration technique goes as follows.
In order to integrate some real-valued function over $\mathcal H^{r,n}_{(d)}$ whose value at some point $h$ is an average over the fiber $V_h$, we lift it to $\mathcal V$ and then pushforward to $\Pn$ using the projections.
The original expected value over $\mathcal H^{r,n}_{(d)}$ is then written as an integral over $\Pn$ which involves the quotient of normal Jacobians of the projections $\pi_1$ and $\pi_2$.
More precisely,

\begin{equation}\label{eq:doblecoarea}
\int_{\mathcal H^{r,n}_{(d)}}\int_{V_h}\phi(h,x)dV_h \ dh=\int_{\Pn}\bigg(\int_{V_x}\phi(h,x)\frac{NJ_{\pi_1}(h,x)}{NJ_{
  \pi_2}(h,x)}dV_x\bigg) \ d\Pn,
\end{equation}
where the quotient of the Normal Jacobians satisfies 
$$
\displaystyle\frac{NJ_{\pi_1}(h,x)}{NJ_{\pi_2}(h,x)}=|\det (Dh(x)Dh(x)^*)|^2
$$
(see Blum et al. \cite{BCSS}*{Section 13.2}).


\subsection{The linear solution variety}

In this section we prove the generalization, to the underdetermined case, of the Proposition $6.6$ of \cite{ABBCS_Evp}, which is essentially the same, but is appended here for completeness.

Consider
$$
\mathcal V^{\text{lin}}=\{(M,v)\in\C^{r\times (n+1)}\times \Pn: \ Mv=0\}.
$$
The linear solution variety $\mathcal V^\text{lin}$, for the underdetermined case, is a $(r+1)n$-dimensional smooth submanifold of $\C^{r\times (n+1)}\times\Pn$, and it inherits the Riemannian structure of the ambient space.

The linear solution variety is equipped with the two canonical projections $\pi^{\text{lin}}_1:\mathcal V^\text{lin}\to \C^{r\times(n+1)}$ and $\pi^{\text{lin}}_2:\mathcal V^\text{lin}\to \Pn.$ 

For $M\in\C^{r\times(n+1)}$, $(\pi^\text{lin}_1)^{-1}(M)$ is a copy of the projective linear subspace corresponding to the kernel of $M$ in $\Pn$, and for $v\in\Pn$, $(\pi^\text{lin}_2)^{-1}(v)$ is a copy of the linear subspace of $\C^{r\times(n+1)}$ consisting of the matrices $A\in\C^{r\times(n+1)}$ such that $Av=0$.
Also, the set of critical points $\Sigma'$ is the set of pairs $(M,v)\in\mathcal V^\text{lin}$ such that $\mathrm{rank} (M)<r$.

In this case the tangent space to $\mathcal V^\text{lin}$ at $(M,v)$ is the set of pairs $(\dot M,\dot v)$ in $\C^{r\times (n+1)}\times\C^{n+1}$ satisfying the following linear equations
$$
\dot Mv+M\dot v=0,\qquad v^*\dot v=0.
$$
Then, if $(M,v)\notin \Sigma'$, for any $\dot v\in\Ker M^\perp$, since $M^\dagger M=id|_{\Ker M^\perp}$, we have
\begin{align*}
  &M\dot v=-\dot M v\\
  &M^\dagger M\dot v=-M^\dagger \dot M v\\
  &\dot v=-M^\dagger \dot M v.
\end{align*}

It is clear then, that we have the following decomposition in orthogonal subspaces,
$$
T_{(M,v)}\mathcal V^{\text{lin}}=\{(0,\dot v): \, \dot v\in \Ker M\}\oplus\{(\dot M,\dot v): \, \dot v= \varphi(\dot M)\}
$$
where $\varphi(\dot M)=-M^\dagger \dot M$.
A routine computation shows that, if $\|v\|=1$, then $\varphi\varphi^*$ is equal to $M^\dagger(M^\dagger)^*$. 
Writing down the singular value decomposition of $M$, it follows that $\det(\varphi\varphi^*)=\det(MM^*)^{-1}$.

Then,
$$
\frac{NJ_{\pi^\text{lin}_1}(M,v)}{NJ_{\pi^\text{lin}_2}(M,v)}=|\det MM^*|
$$
(see Blum et al. \cite{BCSS}*{p. 242}).
\begin{prop}\label{linvar}
Let $\phi:\C^{r\times (n+1)}\to[0,\infty)$ be a measurable unitary invariant function in the sense that $\phi(MU^*)=\phi(M)$ for any unitary matrix $U\in\mathcal U(n+1)$.
Then,
$$
\underset{M\in\C^{r\times (n+1)}}{\mathbb E}(\phi(M))=\frac{\Gamma(n-r+1)}{\Gamma(n+1)}\underset{A\in\C^{r\times n}}{\mathbb E}(\phi((0 | A))\cdot|\det AA^*|),
$$
where $(0|A)$ is the matrix whose first column is all $0$'s and the following columns are the same as $A$.
\end{prop}
\begin{proof}
  Let $\chi(M,v)=\phi(M)e^{-\|M\|^2_F}$.
  Then, applying the double-fibration technique we get that 
  $$
  \int_{M\in\C^{r\times (n+1)}}\int_{v\in\mathbb P(\Ker M)}\chi(M,v)\, d\mathbb P(\Ker M) \, d\C^{r\times(n+1)}
  $$
  is equal to
  $$
  \int_{v\in\Pn}\int_{M:Mv=0}\chi(M,v)\cdot|\det MM^*| \, dM \, d\Pn.
  $$

  Now, since $\chi$ does not depend on $v$, and $\mathbb P(\Ker M)$ is a copy of $\mathbb P(\C^{n-r+1})$, we have
  that the first integral is equal to
  $$
  \frac{\pi^{n-r}}{\Gamma(n-r+1)}\int_{M\in\C^{r\times(n+1)}}\phi(M)e^{-\|M\|^2_F} \, dM.
  $$

  Also, by parametrizing $\{M:Mv=0\}$ by $\{(0|A)U^*_v:A\in\C^{r\times n}\}$, where $U_v$ is any  matrix in $\mathcal U(n+1)$ such that $U_v e_1=v$, and since $\chi(M,v)\cdot|\det MM^*|$ is equal to $\chi(MU^*,Uv)\cdot|\det MU^*UM^*|$ for all $U\in\mathcal U(n+1)$ by hypothesis,
  we have that the second integral is equal to 
  $$
  \frac{\pi^n}{\Gamma(n+1)}\int_{A\in\C^{r\times n}}\phi(0|A)\cdot|\det AA^*|e^{-\|A\|^2_F} \, dM.
  $$
  In conclusion, dividing by $\pi^{rn}$, we get
  $$
  \frac{1}{\Gamma(n-r+1)}\underset{M\in\C^{r\times (n+1)}}{\mathbb E}(\phi(M))=\frac{1}{\Gamma(n+1)}\underset{A\in\C^{r\times n}}{\mathbb E}(\phi((0 | A))\cdot|\det AA^*|)
  $$
\end{proof}

\subsection{The solution variety for the finding kernel problem}
It will be useful to consider a scheme similar to that of the previous section for the case of finding kernels of rectangular matrices.

\vspace{10pt}

We denote by $G(k,l)$ the Grassmannian with of $k$-planes in $\C^l$, i.e. the set of $k$-dimensional linear subspace of $\C^l$. It can be seen that $G(k,l)$ is a projective variety of dimension $k(l-k)$ and degree $\Gamma(k(k-l)+1)\displaystyle\prod_{i=1}^{k}\frac{\Gamma(i)}{\Gamma(k+i)}$, so its volume with regar to the usual Riemannian metric, satisfy
\begin{align}\label{eq:volgrass}
  \vol(G(k,l))=\pi^{k(l-k)}\prod_{i=1}^{k}\frac{\Gamma(i)}{\Gamma(k+i)}
\end{align}  
(see Harris \cite{Harris}, Mumford \cite{Mumford}).

\vspace{10pt}

Now, consider 
$$
\mathcal V^{\text{ker}}=\{(M,V)\in\C^{r\times n}\times G(n-r,n): \ MV=0\}
$$
where $G(n-r,n)$ is the grassmannian of $(n-r)$-planes in $\C^n$.

The linear kernel variety $\mathcal V^\text{ker}$ is an $rn$-dimensional smooth submanifold of $\C^{r\times n}\times G(n-r,n)$, and it inherits the Riemannian structure of the ambient space.

The linear kernel variety is equipped with the two canonical projections $\pi^{\text{ker}}_1$ and $\pi^{\text{ker}}_2.$ 

In this case the tangent space to $\mathcal V^\text{ker}$ at $(M,V)$ is the set of pairs $(\dot M,\dot V)$ in $\C^{r\times n}\times\C^{n\times (n-r)}$ satisfying the following linear equations
$$
\dot MV+M\dot V=0, \qquad V^*\dot V=0.
$$

Then, if $(M,V)\notin\Sigma'$, for any $\dot V$ such that $V^*\dot V$, we have 
\begin{align*}
  &M\dot V=-\dot M V\\
  &M^\dagger M\dot V=-M^\dagger \dot M V\\
  &\dot V=-M^\dagger \dot M V.
\end{align*}

It is clear then, that the tangent space at $(M,V)$ is the set of pairs $(\dot M,\dot V)$ such that $\dot V=\Phi(\dot M)$, where $\Phi(\dot M)=-M^\dagger \dot M V$.
A routine computation shows that $\det(\Phi\Phi^*)=|\det MM^*|^{n-r}$.

It follows that,
$$
\frac{NJ_{\pi_1^{\text{ker}}}(M,V)}{NJ_{\pi_1^{\text{ker}}}(M,V)}=|\det MM^*|^{n-r}.
$$

Applying the double-fibration technique and \eqref{eq:volgrass} we get the following.

\begin{prop}\label{kervar}
Let $\phi:\mathcal V^{\ker}\to[0,\infty)$ be a measurable  unitarily invariant function in the sense that $\phi(M,V)=\phi(MU^*,UV)$ for any unitary matrix $U\in\mathcal U(n)$.
Then,
$$
\underset{M\in\C^{r\times n}}{\mathbb E}(\phi(M,\Ker M))=\prod_{i=1}^{n-r}\frac{\Gamma(i)}{\Gamma(r+i)}\underset{B\in\C^{r\times r}}{\mathbb E}\left(\phi((0 | B),V_{n-r})\cdot|\det B|^{2(n-r)}\right)
$$
where $V_{n-r}$ is the subspace generated by  the first $n-r$ vectors in the canonical base.
\end{prop}

\begin{rmk}
  Observe that the previous proposition is in fact a generalization in another sense than Proposition \ref{linvar} of the same result.
  It is clear from the fact that if $r=n-1$, the previous proposition is exactly Proposition $6.6$ of \cite{ABBCS_Evp}.
\end{rmk}

\begin{rmk}
  Observe that the previous proposition can be proved by applying Proposition \ref{linvar} successively  $n-r$ times.
\end{rmk}

Now, we have everything to prove Theorem \ref{relavg}.

\subsection{Proof of Theorem \ref{relavg}}\label{thmproof}

Let us prove the following proposition, which is a first step towards the general version of Theorem \ref{relavg}.
Consider the relative Frobenius condition number $\hat\mu^r_F$, then the $\alpha$-th moment of the relative condition number is defined as
\begin{equation}\label{eq:relmom}
  \hat\mu^{r,\alpha}_{F}(h)\coloneqq\frac{1}{\vol(V_h)}\int_{x\in V_h}\frac{\mu^r_{F}(h,x)^\alpha}{\|h\|^\alpha} dV_h.  
\end{equation}

\begin{prop}\label{exmuhat} The expected value, with respect to the standard Gaussian distribution given in \eqref{eq:gaussian}, of the $\alpha$-th moment of the relative Frobenius condition number $\hat\mu^{r,\alpha}_F$ satisfies:
$$
\underset{h\in\mathcal H^{r,n}_{(d)}}{\mathbb E}\left(\hat\mu^{r,\alpha}_F(h)\right)=\underset{M\in\C^{r\times (n+1)}}{\mathbb E}\left(\|M^\dagger\|_{F}^\alpha\right).
$$
\end{prop}

\begin{proof}
Since $h\in\mathcal H^{r,n}_{(d)}\setminus\Sigma$, $V_h$ is a projective variety of degree $\mathcal D_r$ and dimension $n-r$, then we have, 
$$
\vol(V_h)=\mathcal D_r\vol(\mathbb P(\C^{n-r+1}))=\mathcal D_r\frac{\pi^{n-r}}{\Gamma(n-r+1)}, \qquad h\in\mathcal H^{r,n}_{(d)}\setminus\Sigma
$$
(see Mumford \cite{Mumford}*{Theorem 5.22}).
Then, by definition \eqref{eq:relmom}, we have the following 
$$
\hat\mu^{r,\alpha}_F(h)=\frac{\Gamma(n-r+1)}{\pi^{n-r}\mathcal D_r}\int_{x\in V_h}\frac{\mu_{F}^r(h,x)^\alpha}{\|h\|^\alpha}dV_h.
$$

Taking expectation with regard to the Gaussian distribution, we get
\begin{align*}
\underset{h\in\mathcal H^{r,n}_{(d)}}{\mathbb E}\left(\hat\mu^{r,\alpha}_F(h)\right)=&\frac{\Gamma(n-r+1)}{\pi^{N+n-r}\mathcal D_r}\int_{h\in\mathcal H^{r,n}_{(d)}}\int_{x\in V_h}\frac{\mu_{F}^r(h,x)^\alpha}{\|h\|^\alpha}e^{-\|h\|^2}dV_h \ dh.
\end{align*}

Then, by the definition of $\mu_{F}^r$ in \eqref{eq:muf}, applying the double-fibration technique \eqref{eq:doblecoarea} and the unitary invariance, taking $x=e_0$, we get that 
$$
\int_{h\in\mathcal H^{r,n}_{(d)}}\int_{x\in V_h}\frac{\mu_{F}^r(h,x)^\alpha}{\|h\|^\alpha}e^{-\|h\|^2}dV_h \ dh
$$
is equal to

\begin{align*}
  &\int_{x\in\Pn}\int_{h\in V_x}\|Dh(x)^\dagger\Delta(\|x\|^{d_i}d_i^{1/2}) \|_{F}^\alpha \cdot |\det (Dh(x)Dh(x)^*)| e^{-\|h\|^2}dV_x \ d\Pn\\
  &=\vol(\Pn)\int_{h\in V_{e_0}}\|Dh(e_0)^\dagger\Delta(d_i^{1/2}) \|_{F}^\alpha \cdot|\det (Dh(e_0)Dh(e_0)^*)| e^{-\|h\|^2}dV_{e_0}\\
  &=\frac{\pi^n \mathcal D_r}{\Gamma(n+1)}\int_{h\in V_{e_0}}\|L_0(h)^\dagger\|_{F}^\alpha \cdot |\det (L_0(h)L_0(h)^*)| e^{-\|h\|^2}dV_{e_0},
\end{align*}
where $L_0(h)$ is given by \eqref{eq:l0}

If we apply \eqref{eq:coarealx}, the latter is equal to 
$$
\frac{\pi^n \mathcal D_r}{\Gamma(n+1)}\int_{A\in\mathbb C^{r\times n}}\|A^\dagger\|_{F}^\alpha \cdot |\det (AA^*)|e^{-\|A\|^2}dA\int_{h\in L_0^{-1}(A)} e^{-\|h\|^2}dL_x^{-1}(A).
$$

Since $L_0^{-1}(A)$ is a linear subspace of dimension $N-r-rn$, the right hand side integral is equal to $\pi^{N-r-rn}$.

Then,
\begin{align*}
  \underset{h\in\mathcal H^{r,n}_{(d)}}{\mathbb E}\left(\hat\mu^{r,\alpha}_F(h)\right)=&\frac{\Gamma(n-r+1)}{\Gamma(n+1)}\underset{A\in\mathbb C^{r\times n}}{\mathbb E}(\|A^\dagger\|_{F}^\alpha \cdot |\det (AA^*)|)\\
  =&\underset{M\in\C^{r\times(n+1)}}{\mathbb E}\left(\|M^\dagger\|_{F}^\alpha\right).
\end{align*}

The last equality follows from Proposition \ref{linvar} applied to the unitarily invariant map $\phi(M)=\|M^{\dagger}\|^\alpha_{F}.$
\end{proof}

\begin{rmk}
  This result, can be seen as a generalization of computations done by Beltrán-Pardo in \cite{BelParLinear}, namely that $(f,x)\mapsto (\Delta(d_i^{-1/2})Df(x),x)$ gives a partial isometry from $\mathcal V$ to the linear variety $\mathcal V^{\text{lin}}$.
\end{rmk}

\begin{rmk}
  Note that $\|M^\dagger\|_F=\left(\sum \sigma_i(M)^2\right)^{1/2}$, where $\sigma_i(M)$ are the singular values of the matrix $M$.
  According to Edelman \cite{Edelman}*{Formula 3.12}, we get that if $\alpha<2(n-r+2)$ the $\alpha$-th moment of $\|M^\dagger\|_F$ is finite.
\end{rmk}

\vspace{10pt}

\begin{teo}\label{exmualpha} Let $0<\alpha<2(n-r+2)$, then the expected value, with respect to the standard Gaussian distribution, of the $\alpha$-th moment of the relative Frobenius condition number $\hat\mu^{r,\alpha}_{F}$ is finite and satisfies:
  $$
  \underset{h\in\mathcal H^{r,n}_{(d)}}{\mathbb E}(\hat\mu^{r,\alpha}_{F}(h))=\mathcal C_{r,n}\underset{M\in\C^{r\times r}}{\mathbb E}\left(\|B^{-1}\|_{F}^\alpha\cdot|\det B|^{2(n-r+1)}\right),
  $$
  where $\mathcal C_{r,n}\coloneqq \displaystyle\prod_{i=1}^{n-r+1}\frac{\Gamma(i)}{\Gamma(r+i)}.$
\end{teo}
\begin{proof}
  
  Applying Proposition \ref{exmuhat} and Proposition \ref{kervar}, we have 
  $$
  \underset{h\in\mathcal H^{r,n}_{(d)}}{\mathbb E}(\hat\mu^{r,\alpha}_{F}(h))=\prod_{i=1}^{n-r+1}\frac{\Gamma(i)}{\Gamma(r+i)}\underset{B\in\C^{r\times r}}{\mathbb E}(\|B^{-1}\|_{F}^\alpha\cdot|\det B|^{2(n-r+1)}).
  $$

  Taking $\mathcal C_{r,n}=\displaystyle\prod_{i=1}^{n-r+1}\frac{\Gamma(i)}{\Gamma(r+i)}$, we get
  $$
  \underset{h\in\mathcal H^{r,n}_{(d)}}{\mathbb E}(\hat\mu^{r,\alpha}_{F}(h))=\mathcal C_{r,n}^\alpha\underset{B\in\C^{r\times r}}{\mathbb E}\left(\|B^{-1}\|_{F}^\alpha\cdot|\det B|^{2(n-r+1)}\right).
  $$
\end{proof}

\begin{rmk}[Proof of Theorem \ref{relavg}]
Consider the case where $\alpha=2$.

Applying Lemma \ref{invnor2mdet} we have that 
$$
\underset{B\in\C^{r\times r}}{\mathbb E}(\|B^{-1}\|_{F}^2\cdot|\det B|^{2(n-r+1)})=\frac{r}{n-r+1}\prod_{i=1}^{r}\frac{\Gamma(n-r+1+i)}{\Gamma(i)}.
$$

It follows
\begin{align*}
\underset{h\in\mathcal H^{r,n}_{(d)}}{\mathbb E}(\hat\mu^{r,2}_{F}(h))&=\left(\prod_{i=1}^{n-r+1}\frac{\Gamma(i)}{\Gamma(r+i)}\right)\left(\frac{r}{n-r+1}\prod_{i=1}^{r}\frac{\Gamma(n-r+1+i)}{\Gamma(i)}\right)\\
&=\frac{r}{n-r+1}
\end{align*} \qed
\end{rmk}

Now, given $h\in\mathcal H^{r,n}_{(d)}\setminus \Sigma$, recall that the $\alpha$-th moment of the (absolute) normalized condition number is
$$
\mu_{F,Av}^{r,\alpha}(h)=\displaystyle\frac{\Gamma(n-r+1)}{\pi^{n-r}\mathcal D_r}\int_{x\in V_h}\mu_{F}^r(h,x)^\alpha dx.
$$

\begin{proof}[Proof of Corolary \ref{avg}]

To compute the expected value of $\mu_{F,Av}^{r,\alpha}$, we just need to apply Lemma 2 of \cite{ABBCS} to $\hat \mu^{r,\alpha}_F$ for $p=-\alpha$, and we get
  $$
  \underset{h\in\mathcal H^{r,n}_{(d)}}{\mathbb E}(\hat\mu^{r,\alpha}_F(h))=\frac{\Gamma(N-\alpha/2)}{\Gamma(N)}\underset{h\in \mathbb S(\mathcal H^{r,n}_{(d)})}{\mathbb E}(\hat\mu^{r,\alpha}_F(h)),
  $$
  where $\mathbb S(\mathcal H^{r,n}_{(d)})$ is the set of $h\in\mathcal H^{r,n}_{(d)}$ such that $\|h\|=1$.
  Since $\hat\mu^{r,\alpha}_F(h)=\mu_{F,Av}^{r,\alpha}(h)$ if $\|h\|=1$, and the latter is scale invariant, we have that 
  $$
  \underset{h\in\mathcal H^{r,n}_{(d)}}{\mathbb E}(\mu_{F,Av}^{r,\alpha}(h))=\frac{\Gamma(N)}{\Gamma(N-\alpha/2)}\underset{h\in\mathcal H^{r,n}_{(d)}}{\mathbb E}(\hat\mu^{r,\alpha}_F(h)).
  $$

\end{proof}

\begin{rmk}
Observe that if one considers the operator norm instead of the Frobenius one in \eqref{eq:muf}, one gets the operator analog of Lemma \ref{exmuhat} and Theorem \ref{exmualpha}.
\end{rmk}

\section{Some useful computations}

In this section we will do some computations which are needed for the results of the previous sections.

\begin{lemma}\label{espnorm}
  Let $v$ be a standard Gaussian random vector in $\Cn$ and $\alpha\in\R$ with $\alpha>-2n$.
  Then,
  $$
  \underset{v\in\Cn}{\mathbb E}\left(\|v\|^{\alpha}\right)=\frac{\Gamma(n+\alpha/2)}{\Gamma(n)}.
  $$
\end{lemma}

\begin{proof}
  By a simple calculation, taking polar coordinates, we have
  \begin{align*}
    \underset{v\in\Cn}{\mathbb E}(\|v\|^\alpha)&=\frac{1}{\pi^n}\int_{v\in\C^n}\|v\|^\alpha e^{-\|v\|^2}d\C\\
      &=\frac{\vol(S^{2n-1})}{\pi^n}\int_0^\infty \rho^{2n+\alpha-1} e^{-\rho^2}d\rho=\frac{\Gamma(n+\alpha/2)}{\Gamma(n)}
  \end{align*}
\end{proof}

\begin{lemma}\label{espnormrest}
  Let $v$ be a standard Gaussian random vector in $\Cn$ and $\alpha,\beta\in\R$ with $2\alpha+\beta>1-2n$.
  Then,
  $$
  \underset{v\in\Cn}{\mathbb E}\left(\|v\|^{2\alpha}\|\Pi_{e_n^\perp} v\|^\beta\right)=\frac{\Gamma(n+\alpha+\beta/2)}{n\Gamma(n-1)}.
  $$
\end{lemma}

\begin{proof}
  By definition of the norm in $\Cn$ and the binomial expansion, we have that 
  $$
  \|v\|^{2\alpha}=\left(\|\Pi_{e_n^\perp}v\|^2+|v_n|^2\right)^{\alpha}=\sum_{i=0}^{\alpha}{\alpha \choose i}(\|\Pi_{e_n^\perp}v\|^2)^{\alpha-i}(|v_n|^2)^i.
  $$
Since the entries of $v$ are independent, it follows
  \begin{align*}
    \underset{v\in\Cn}{\mathbb E}\left(\|v\|^{2\alpha}\|\Pi_{e_n^\perp} v\|^\beta\right)&=\mathbb E\left(\sum_{i=1}^{\alpha}{\alpha \choose i}\left(\|\Pi_{e_n^\perp}v\|^2\right)^{\alpha+\beta/2-i}\left(|v_n|^2\right)^i\right)\\
    &=\sum_{i=1}^{\alpha}{\alpha \choose i}\underset{w\in\C^{n-1}}{\mathbb E}\left(\|w\|^{2\alpha+\beta-2i}\right)\underset{z\in\C}{\mathbb E}\left(|z|^{2i}\right).
  \end{align*}
    
  Then, applying the previous lemma

  \begin{align*}
    \underset{v\in\Cn}{\mathbb E}\left(\|v\|^{2\alpha}\|\Pi_{e_n^\perp} v\|^\beta\right)&=\sum_{i=0}^{\alpha}\frac{\Gamma(\alpha+1)\Gamma(n+\alpha+\beta/2-i-1)}{\Gamma(\alpha-i+1)\Gamma(n-1)}\\
      &=\frac{\Gamma(\alpha+1)}{\Gamma(n-1)}\frac{\Gamma(n+\alpha+\beta/2)}{n\Gamma(\alpha+1)}\\
      &=\frac{\Gamma(n+\alpha+\beta/2)}{n\Gamma(n-1)}.
  \end{align*}
\end{proof}

The following is a generalization of the Proposition $7.1$ of Armentano et al. \cite{ABBCS_Evp}.

\begin{lemma}\label{invnor2mdet}
  Let $A$ be a standard Gaussian random matrix in $\C^{r\times r}$ and $k>0$.
  Then,
  $$
  \underset{A\in\C^{r\times r}}{\mathbb E}\left(\|A^{-1}\|_F^2\cdot|\det A|^{2k}\right)=\frac{r}{k}\prod_{i=1}^{r}\frac{\Gamma(k+i)}{\Gamma(i)}.
  $$
\end{lemma}

\begin{proof}
  Observe that from a direct application of Cramer's rule we have 
  $$
  \|A^{-1}\|_F^2=\frac{1}{|\det A|^2}\sum_{i,j=1}^{r}|\det (A^{ij})|^2.
  $$
  It follows,
  \begin{align*}
    \underset{A\in\C^{r\times r}}{\mathbb E}\left(\|A^{-1}\|_F^2\cdot|\det A|^{2k}\right)&=\mathbb E\left(\sum_{i,j=1}^{r}|\det(A^{ij})|^2\cdot|\det A|^{2(k-1)}\right)\\
    &=r^2\mathbb E\left(|\det(A^{nn})|^2\cdot|\det A|^{2(k-1)}\right).
  \end{align*}

  Using the ideas of Azaïs-Wschebor \cite{AzaisWschebor}, considering $|\det A|$ as the volume of the parallelepiped generated by the columns of $A$ we get, 
  $$
  \underset{A\in\C^{r\times r}}{\mathbb E}\left(\|A^{-1}\|_F^2\cdot|\det A|^{2k}\right)=r^2 \underset{z\in\C}{\mathbb E}\left(|z|^{2(k-1)}\right)\cdot\prod_{i=2}^r \underset{v\in\C^i}{\mathbb E}\left(\|v\|^{2(k-1)}\|\Pi_{e_i^\perp} v\|^2\right).
  $$
  
  Then, applying the previous lemmas we get
  \begin{align*}
    \underset{A\in\C^{r\times r}}{\mathbb E}\left(\|A^{-1}\|_F^2\cdot|\det A|^{2k}\right)&=r^2\frac{\Gamma(k)}{\Gamma(1)}\prod_{i=2}^{r}\frac{\Gamma(i+k)}{i\Gamma(i-1)}\\
    &=r\frac{\Gamma(k)}{\Gamma(k+1)}\prod_{j=1}^{r}\frac{\Gamma(k+j)}{\Gamma(j)}=\frac{r}{k}\prod_{j=1}^{r}\frac{\Gamma(k+j)}{\Gamma(j)}.
  \end{align*}

\end{proof}

\end{document}